\pdfoutput=1
\documentclass[12pt]{amsart}
\usepackage[english]{babel}
\usepackage[utf8]{inputenc}
\usepackage[T1]{fontenc}
\usepackage{amssymb,amsmath,amsthm,amsfonts}
\usepackage{lmodern}
\usepackage{thmtools,mathtools}
\usepackage{microtype}
\usepackage{tikz,tikz-cd}
\usepackage{float}
\usepackage[breaklinks,pdfencoding=auto,psdextra]{hyperref}
\usepackage[margin=1in]{geometry}
\theoremstyle{plain}
\newtheorem{theorem}{Theorem}[section]
\newtheorem{lemma}[theorem]{Lemma}
\newtheorem{corollary}[theorem]{Corollary}
\newtheorem{proposition}[theorem]{Proposition}
\theoremstyle{definition}
\newtheorem{definition}[theorem]{Definition}

\newtheorem{remark}[theorem]{Remark}

\def\bC{{\mathbf{C}}}

\def\bP{{\mathbf{P}}}
\def\bQ{{\mathbf{Q}}}
\def\bR{{\mathbf{R}}}

\def\bZ{{\mathbf{Z}}}

\def\cB{{\mathcal{B}}}
\def\cC{{\mathcal{C}}}

\def\cH{{\mathcal{H}}}

\def\cK{{\mathcal{K}}}

\def\cM{{\mathcal{M}}}

\def\cP{{\mathcal{P}}}

\def\cX{{\mathcal{X}}}

\def\fM{{\mathfrak{M}}}

\def\Aut{\operatorname{Aut}}

\def\div{\operatorname{div}}

\def\Id{\operatorname{Id}}
\def\Im{\operatorname{Im}}

\def\Mon{\operatorname{Mon}}
\def\NS{\operatorname{NS}}
\def\O{\operatorname{O}}

\def\Re{\operatorname{Re}}

\def\SO{\operatorname{SO}}

\def\alg{{\mathrm{alg}}}

\def\KKK{{\mathrm{K3}}}
\def\Kum{{\mathrm{Kum}}}

\let\ordexists\exists
\def\exists{\operatorname{\ordexists}}
\let\ordforall\forall
\def\forall{\operatorname{\ordforall}}

\def\inner#1{{\left<{#1}\right>}}
\def\set#1{{\left\{{#1}\right\}}}
\def\setmid#1#2{{\left\{{#1}\;\middle|\;{#2}\right\}}}

\def\tilde{\widetilde}
\def\setminus{\smallsetminus}
\def\emptyset{\varnothing}

\def\cf{{\it cf.\ }}

\def\longarrow#1#2{\mathchoice{#2}{#1}{#1}{#1}}
\def\to{\longarrow{\rightarrow}{\longrightarrow}}
\def\simto{\longarrow{\xrightarrow\sim}{\stackrel\sim\longrightarrow}}
\def\into{\longarrow{\hookrightarrow}{\lhook\joinrel\longrightarrow}}

\let\shortmapsto\mapsto
\def\mapsto{\longarrow{\shortmapsto}{\longmapsto}}
\def\tO{\widetilde{\O}}
\def\tSO{\widetilde{\SO}}
\def\hO{\widehat{\O}}
\def\tL{\widetilde{\Lambda}}

\def\amp{\mathrm{amp}}
\def\marked{\mathrm{marked}}
\def\OG{\mathrm{OG}}
\title[Image of the period map for hyperkähler manifolds]{On the image of the period map for polarized hyperkähler manifolds}

\author[J.~Song]{Jieao Song}
\address{Université Paris Cité, CNRS, IMJ-PRG, F-75013 Paris, France}
\email{\href{mailto:jieao.song@imj-prg.fr}{\tt jieao.song@imj-prg.fr}}
\date{\today}

\begin{document}
\begin{abstract}
The moduli space for polarized hyperkähler manifolds of~$\KKK^{[m]}$-type or
$\Kum_m$-type with a given polarization type is not necessarily connected,
which is a phenomenon that only happens for~$m$ large.
The period map restricted to each connected component
gives an open embedding into the period domain, and the complement of the image
is a finite union of Heegner divisors. We give a simplified formula for the
number of connected components, as well as a simplified criterion to enumerate
the Heegner divisors in the complement. In particular, we show that the image
of the period map may be different when restricted to different components of
the moduli space.
\end{abstract}
\maketitle

\section{Introduction}\label{sec:introduction}
A \emph{hyperkähler} manifold is a simply-connected compact Kähler manifold~$X$
such that $H^0(X,\Omega_X^2)=\bC\omega$, where~$\omega$ is a nowhere degenerate
holomorphic 2-form on~$X$. These manifolds form an important class among
compact Kähler manifolds with trivial canonical bundle. For example, in
dimension~$2$, these are precisely \emph{K3 surfaces}. Higher-dimensional
examples include manifolds of~$\KKK^{[m]}$-type (those deformation equivalent to
Hilbert powers of K3 surfaces), $\Kum_m$-type (generalized Kummer varieties and
their deformations), and two families of examples,~$\OG_6$ and~$\OG_{10}$,
discovered by O'Grady. These four families of examples are the only deformation
types known up to now.

Given a hyperkähler manifold~$X$, there is a quadratic form called the
\emph{Beauville--Bogomolov--Fujiki form} on the free abelian cohomology
group~$H^2(X,\bZ)$. It gives us a lattice structure of signature~$(3,b_2-3)$
on the cohomology group and consequently, a polarized Hodge structure, which is
fundamental in the study of hyperkähler manifolds. In the case of
a K3 surface, this form coincides with the intersection product. If we fix the
deformation type of~$X$, the lattice is also fixed and will be denoted
by~$\Lambda$.  We call an isometry $\eta\colon H^2(X,\bZ)\simto\Lambda$ a
\emph{marking} of~$X$. Denote by~$\cM_\marked$ the moduli space for marked
hyperkähler manifolds~$(X,\eta)$ of the given deformation type. On each
connected component~$\cM_\marked^0$ of the moduli space~$\cM_\marked$, the
Hodge structures provide a \emph{period map}
\[
\wp^0_\marked\colon \cM^0_\marked\to\Omega_\marked,
\]
where
\[
\Omega_\marked\coloneqq\setmid{[x]\in \bP(\Lambda_\bC)}{(x,x)=0,(x,\bar x)>0}
\]
is a complex manifold called the \emph{period domain}.
The global Torelli theorem, proven by Verbitsky, states
that~$\wp^0_\marked$ is surjective, generically injective, and
identifies pairwise inseparable points.

On a projective hyperkähler manifold~$X$, we may consider the extra datum of a
\emph{polarization}, that is, a primitive ample class~$H\in H^2(X,\bZ)$. Any
marking~$\eta$ maps~$H$ to a vector~$\eta(H)\in\Lambda$, so it is reasonable to
define the \emph{polarization type}~$T$ of $(X,H)$ as the~$\O(\Lambda)$-orbit
of~$\eta(H)$ in~$\Lambda$, which does not depend on the choice of the
marking~$\eta$. There is a quasi-projective moduli space~$\cM_T$ for
polarized hyperkähler manifolds $(X,H)$ of fixed polarization type~$T$. For K3
surfaces, each polarization type~$T$ is uniquely determined by its square~$2d$
and each moduli space~$\cM_{2d}$ is an irreducible quasi-projective variety of
dimension~19. However, for their higher-dimensional analogues, the polarization
types are more complicated to describe: apart from the square, there is another
invariant, the \emph{divisibility}. Moreover, Apostolov showed in~\cite{Apostolov}
that for some polarization types~$T$ on manifolds of~$\KKK^{[m]}$-type, the moduli
space~$\cM_T$ may have several connected components. Onorati obtained similar
results for $\Kum_m$-type in \cite{Onorati:kummer}. We shall review their results
and give a simplified expression for the exact number of components in
Section~\ref{sec:number} (Proposition~\ref{prop:numberoftau} and
Proposition~\ref{prop:numberofT}).

One can also consider the period map for polarized hyperkähler manifolds and
its restriction to each connected component~$\cM^0_T$ of the
polarized moduli space~$\cM_T$. We will use the letter~$\tau$ to denote a
\emph{deformation type} of polarizations of type~$T$. Such deformation types
are in bijection with the connected components of~$\cM_T$, so we will
write~$\cM_\tau$ instead of~$\cM^0_T$. In order to get rid of the choice of a
marking, we consider the quotient of the corresponding period domain~$\Omega$,
which is a hyperplane section inside~$\Omega_{\marked}$, by the action of the
elements in the orthogonal group~$\O(\Lambda)$ that stabilize the polarization.
In this way, we get a period domain~$\cP_T$, depending only on the polarization
type~$T$. But the global Torelli theorem no longer holds in this case, as the
map from~$\cM_\tau$ to~$\cP_T$ might not be
generically injective. In fact, it factors through~$\cP_\tau$, the quotient
of~$\Omega$ by a smaller group~$\Mon(\Lambda)$, the \emph{monodromy group},
which is a normal subgroup of~$\O(\Lambda)$ for all the
known deformation types (see Table~\ref{table:mon}).
Thus the correct global Torelli theorem says that the polarized period map
\[\begin{tikzcd}
\wp_\tau\colon\cM_\tau\rar[hook]&\cP_\tau\dar{/G}\\ &\cP_T
\end{tikzcd}\]
is an open immersion, where~$\cP_\tau$ is a covering space of~$\cP_T$ with
finite deck transformation group~$G$. The complement of the image of this open
immersion is a finite union of divisors in~$\cP_\tau$. Intuitively, when the
periods of the manifolds~$X$ in the family move towards the boundary of the
image, the polarization~$H$ on~$X$ will move towards the boundary of the ample
cone.  Therefore, the determination of the divisors in the complement of the
image is intimately related to the geometry of the ample cone for manifolds~$X$
in the family.

In the~$\KKK^{[m]}$-type case, the description of the ample cone was given by
Bayer--Hassett--Tschinkel~\cite{BayerHassettTschinkel}, using the theory of
Bayer--Macrì~\cite{BayerMacri}. The description is based on a canonical embedding
of~$H^2(X,\bZ)$ into a larger lattice~$\tilde\Lambda$, known as the \emph{Mukai
lattice}. The ample cone can then be described using some numerical conditions.
The analogous result for $\Kum_m$-type was obtained by Yoshioka \cite{Yoshioka}.
We will review this in Section~\ref{sec:image} and give a simplified
description, without explicitly referring to the larger Mukai
lattice (Proposition~\ref{prop:ample_cone}). We will use this description to
characterize the divisors in the complement of the image of the period map.
Note that the~$\KKK^{[2]}$-type case was completely treated
in~\cite{DebarreMacri}
(see also \cite[Appendix~B]{Debarre:survey}).

A natural question arises of whether for a given polarization type~$T$,
different connected components~$\cM_\tau$ of~$\cM_T$ have the same image
in~$\cP_\tau$ under their corresponding period map. This question in general
is not well-posed, as there is no canonical way to identify the period
domains~$\cP_\tau$ for different components, due to the action of the
deck transformation group~$G$. Nevertheless, there is no problem of
identification when~$G$ is trivial, and we provide a negative answer in
the~$\KKK^{[m]}$-type case:
by using our numerical description of the image, we construct in
Section~\ref{sec:examples} an example where two connected components of the
same~$\cM_T$ have different images in~$\cP_T$. We will also give another
example where the group~$G$ is non-trivial and the
image of the period map in~$\cP_\tau$ is not~$G$-invariant above~$\cP_T$.

\subsection*{Notation} For a fixed deformation type of hyperkähler
manifolds, we use~$\cM_{\marked}$ (resp.~$\cM_T$) to denote the marked
(resp.~polarized) moduli space. The notation~$\cM^0$ will be used to denote
a connected component of the corresponding moduli space~$\cM$.

For a positive integer~$n$, we denote by~$\rho(n)$ the number of distinct prime
divisors of~$n$ and by~$\tilde\rho(n)$ the number~$\rho(n)$ if~$n$ is odd
and~$\rho(n/2)$ if~$n$ is even. For a prime number $p$, we write $v_p(n)$ for
the $p$-adic valuation of $n$.

To treat $\KKK^{[m]}$-type and $\Kum_m$-type manifolds simultaneously, we let
$\tilde m=m-1$ for $\KKK^{[m]}$-type and $\tilde m=m+1$ for $\Kum_m$-type.

\subsection*{Acknowledgements}
This text results from my master thesis. I would like to thank my advisor
Olivier Debarre for introducing me to this subject, and for carefully reading
the manuscript and providing many corrections and suggestions. I also
heartily thank the referees, in particular for suggesting that various results
can be extended to all known deformation types.

\section{Setup}\label{sec:setup}
In this section, we review the construction of the polarized period map and its
relation with the monodromy group, following the work of
Markman~\cite[Section 4,7, and 8]{Markman:survey}. We reformulate some of the results
to give a simpler presentation and to better suit our needs for later sections.
We will consider a fixed deformation type of hyperkähler manifolds and denote
by~$\Lambda$ the lattice defined by the Beauville--Bogomolov--Fujiki form on
the second cohomology group, which has signature~$(3,b_2-3)$.

First we recall the following definitions (\cf \cite[Definition
1.1]{Markman:survey}).
\begin{definition}
Let~$X$ and~$X'$ be hyperkähler manifolds of the given deformation type.
\begin{itemize}
\item[(i)] An isomorphism $f\colon H^2(X,\bZ)\simto H^2(X',\bZ)$ is called a
\emph{parallel transport operator} if there exist a smooth and proper family
$\pi\colon\cX\to B$ of hyperkähler manifolds, with points $b,b'\in B$
and a path $\gamma\colon[0,1]\to B$ connecting~$b$ and~$b'$, such that $X\simeq
\cX_b$, $X'\simeq\cX_{b'}$, and~$f$ is given as the parallel transport in the
local system $R^2\pi_*\bZ$ along~$\gamma$.
\item[(ii)] An automorphism $f\colon H^2(X,\bZ)\simto H^2(X,\bZ)$ that is a
parallel transport operator is called a \emph{monodromy operator}. The subgroup
of $\O(H^2(X,\bZ))$ generated by monodromy operators is called the
\emph{monodromy group} of~$X$ and denoted by~$\Mon(X)$.
\item[(iii)] If $(X,H)$ and $(X',H')$ are polarized hyperkähler manifolds, we
define similarly a \emph{polarized parallel transport operator} $f\colon
H^2(X,\bZ)\simto H^2(X',\bZ)$ to be one induced by a path~$\gamma$ in a family
of polarized hyperkähler manifolds. In other words, the local system
$R^2\pi_*\bZ$ admits a section~$h$ of ample classes, such that $h(b)=H$ and
$h(b')=H'$.\qedhere
\end{itemize}
\end{definition}

In this paper, we will make the assumption that the monodromy group~$\Mon(X)$
is a normal subgroup of~$\O(H^2(X,\bZ))$, in which case it can be identified as
a subgroup~$\Mon(\Lambda)$ of~$\O(\Lambda)$. This holds for all
known deformation types of hyperkähler manifolds.

A first property of the monodromy group~$\Mon(\Lambda)$ can be given in terms
of the \emph{spinor norm}, which is the following homomorphism of groups
\[
\sigma\colon\O(\Lambda_\bR)\simeq\O(3,b_2-3)\to\set{\pm1},
\]
given by the action on the orientation of a positive three-space~$W_3$
of~$\Lambda_\bR$. In a more canonical way, we may consider the positive cone
\[
\tilde\cC_\Lambda\coloneqq\setmid{x\in\Lambda_\bR}{(x,x)>0}.
\]
For any positive three-space~$W_3$ in~$\Lambda_\bR$,~$W_3\setminus\set{0}$ is a
deformation retract of~$\tilde\cC_\Lambda$. So an orientation of~$W_3$
determines a generator of~$H^2(W_3\setminus\set{0},\bZ)\simeq
H^2(\tilde\cC_\Lambda,\bZ)\simeq\bZ$. The two generators
of~$H^2(\tilde\cC_\Lambda,\bZ)$ are called \emph{orientation classes} of the
positive cone~$\tilde\cC_\Lambda$ and the spinor norm can be defined by the
action on them (\cf \cite[Section 4]{Markman:survey}).
For any subgroup~$G$ of~$\O(\Lambda)$, we write~$G^+$ for the subgroup of~$G$
consisting of elements of trivial spinor norm.

\begin{proposition}\label{prop:Oplus}
The monodromy group~$\Mon(\Lambda)$ is contained in~$\O^+(\Lambda)$.
\end{proposition}
\begin{proof}
For a marked pair~$(X,\eta)$ with period~$[x]\in\Omega_\marked$,
we can take a Kähler class~$H$ on~$X$ and consider the orientation on the
positive three-space~$\bC x\oplus \bR\eta(H)$ given by the basis~$\set{\Re
x,\Im x, \eta(H)}$. This gives a distinguished orientation class
of~$\tilde\cC_\Lambda$, which is constant on each connected
component~$\cM^0_\marked$ of the marked moduli space~$\cM_\marked$. Therefore
every monodromy operator must have trivial spinor norm.
\end{proof}

From now on, we pick one connected component~$\cM^0_\marked$ of the marked
moduli space $\cM_\marked$. Recall from the introduction that we have the
period map
\begin{equation}\label{eq:marked-period}
\wp=\wp^0_\marked\colon \cM^0_\marked\to\Omega_\marked,
\end{equation}
which is surjective by the global Torelli theorem.
Let~$h\in \Lambda$ be a primitive element of positive square.  Consider the
hyperplane section
\begin{align*}\label{eq:Omega_h}
\Omega_{\marked}\cap h^\perp=&\setmid{[x]\in \Omega_{\marked}}{(x,h)=0}\\
=&\setmid{[x]\in\bP(\Lambda_\bC)} {(x,x)=(x,h)=0,(x,\bar x)>0}
\end{align*}
inside the marked period domain~$\Omega_\marked$.  It has two connected
components denoted by~$\Omega_{h}$ and~$\Omega_{-h}$.  For
any~$[x]\in\Omega_h\sqcup\Omega_{-h}$, the real vector space~$\bC x\oplus
\bR h$ is a positive three-space in~$\Lambda_\bR$, but the orientation classes
given by the basis~$\set{\Re x,\Im x,h}$ are opposite on the two connected
components. Since there is a distinguished orientation class for the connected
component~$\cM^0_\marked$, up to interchanging~$\Omega_h$ and~$\Omega_{-h}$, we
may suppose that it coincides with~$\set{\Re x,\Im x,h}$ for~$[x]\in\Omega_h$
(and consequently it also coincides with~$\set{\Re x,\Im x, -h}$
for~$[x]\in\Omega_{-h}$).

Consider the preimages under the period map~\eqref{eq:marked-period} of each
of these two connected components. We denote them by~$\cM_{h}$ and~$\cM_{-h}$.
Due to the surjectivity of the period map, both are non-empty divisors
in~$\cM^0_\marked$. In fact, the union~$\cM_h\sqcup\cM_{-h}$ is
exactly the locus where the class~$\eta^{-1}(h)$ is algebraic.
\begin{proposition}\label{prop:Omega_h}
For a very general~$(X,\eta)$ in~$\cM_{h}$, the
class~$\eta^{-1}(h)$ is ample, while for a very general~$(X,\eta)$
in~$\cM_{-h}$, the class~$\eta^{-1}(-h)$ is ample.
\end{proposition}
\begin{proof}
For a very general element~$(X,\eta)$ in~$\cM_{h}$ with period~$[x]\in\Omega_h$,
the Néron--Severi group is generated by the class~$H\coloneqq\eta^{-1}(h)$.
In this case the Kähler cone coincides with the positive
cone~\cite[Corollary~7.2]{Huybrechts:basic}.
Since~$h$ is primitive of positive square, this implies that either~$H$ or~$-H$
is ample. On the other hand,~$[x]$ lies in~$\Omega_h$, so the orientation class
given by~$\set{\Re x,\Im x,h}$ coincides with the distinguished one, which can
be given by~$\set{\Re x,\Im x,\eta(H')}$ for some Kähler class~$H'$. This
implies that only~$H$ can be ample.  By symmetry, we get the result for~$-h$.
\end{proof}
By removing the locus inside~$\cM_h$ where~$\eta^{-1}(h)$ is not ample, which
is a countable union of closed complex analytic subsets, we get the following result
\cite[Corollary 7.3]{Markman:survey}.
\begin{proposition}[Markman]
\label{prop:M_amp_dense}
Let~$\cM^\amp_{h}$ be the locus in~$\cM_{h}$ that consists of marked
pairs~$(X,\eta)$ such that~$\eta^{-1}(h)$ is ample. Then~$\cM^\amp_{h}$ is
connected and Hausdorff, and the marked period map $\wp$ restricts to an
injective map from $\cM^\amp_h$ onto a dense open subset of~$\Omega_h$ (in
the analytic topology).
\end{proposition}

\begin{remark}
In Markman's survey, the domains~$\Omega_{h}$,~$\cM_{h}$,
and~$\cM^\amp_{h}$ are denoted as~$\Omega^+_{h^\perp}$, $\fM^+_{h^\perp}$,
and~$\fM^a_{h^\perp}$.  We believe our notation is simpler and better reflects
the symmetry between~$h$ and~$-h$: we may
identify~$\Omega^-_{h^\perp}=\Omega^+_{(-h)^\perp}$ as~$\Omega_{-h}$,
and~$\fM^-_{h^\perp}=\fM^+_{(-h)^\perp}$ as~$\cM_{-h}$.
\end{remark}

The connectedness of the locus~$\cM^\amp_h$ implies the following
result {\cite[Corollary 7.4]{Markman:survey}}, which determines whether two
polarized hyperkähler manifolds lie in the same connected component of the
polarized moduli space.
\begin{proposition}[Markman]
\label{prop:parallel}
A parallel transport operator
\[
f\colon H^2(X,\bZ)\simto H^2(X',\bZ)
\]
is a polarized parallel transport operator from $(X,H)$ to $(X',H')$ if and
only if $f(H)=H'$.
\end{proposition}

\begin{definition}
We fix one connected component~$\cM_\marked^0$ of the marked moduli
space $\cM_\marked$ as before. Given a polarized pair~$(X,H)$, choose a
marking~$\eta$ such that~$(X,\eta)$ lies in~$\cM_\marked^0$. We define the
\emph{polarization type}~$T$ of~$(X,H)$ to be the~$\O(\Lambda)$-orbit
of~$\eta(H)$ in~$\Lambda$. We also denote by~$\tau$
the~$\Mon(\Lambda)$-orbit of~$\eta(H)$ in~$\Lambda$, which is contained in~$T$.
This orbit is clearly constant on each connected component~$\cM^0_T$
of~$\cM_T$, so we have a map
\begin{equation}\label{eq:bijection}
\set{\text{connected components of }\cM_T}\to\set{\Mon(\Lambda)\text{-orbits
contained in }T},
\end{equation}
which may depend on the initial choice of the connected
component~$\cM_\marked^0$.  We will call the orbit~$\tau$ the \emph{deformation
type} of~$(X,H)$.
\end{definition}

Proposition~\ref{prop:parallel} can be used to show that the deformation type defined
here is the good notion.  More precisely, we have the following result.

\begin{proposition}\label{prop:deftype}
Let~$T$ be a polarization type, in other words, an~$\O(\Lambda)$-orbit of
a primitive element of positive square. The map \eqref{eq:bijection} above
gives a bijection from the set of connected components of~$\cM_T$ to the set
of~$\Mon(\Lambda)$-orbits contained in~$T$.
\end{proposition}
\begin{proof}
For the injectivity, suppose that two polarized pairs~$(X,H)$ and~$(X',H')$
have the same deformation type, which means that we may choose markings~$\eta$
and~$\eta'$ such that~$(X,\eta)$ and~$(X',\eta')$ both lie in the fixed
connected component~$\cM_\marked^0$, and~$\eta(H)$ and~$\eta'(H')$ have the
same~$\Mon(\Lambda)$-orbit in~$\Lambda$. We want to show that~$(X,H)$
and~$(X',H')$ lie in the same connected component of~$\cM_T$.

Suppose that there exists some $\phi\in\Mon(\Lambda)$ such that
$\phi\circ\eta(H)=\eta'(H')$. By the definition of~$\Mon(\Lambda)$, the
marking~$(X,\phi\circ\eta)$ is also in~$\cM_\marked^0$. The isomorphism
$\eta'^{-1}\circ\phi\circ\eta$ is a parallel transport operator that takes~$H$
to~$H'$ so, by Proposition~\ref{prop:parallel}, it is a polarized one, that is,
$(X,H)$ and $(X',H')$ are indeed connected by some path in the polarized moduli
space~$\cM_T$.

For the surjectivity, since the locus~$\cM^\amp_h$ is non-empty for
every~$h\in T$, the class~$h$ can always be realized as the image~$\eta(H)$ for
some polarized pair~$(X,H)$ and a marking~$\eta$ with~$(X,\eta)$ lying in the
fixed connected component~$\cM^0_\marked$. This in particular means that
every~$\Mon(\Lambda)$-orbit can be realized as the deformation type of some
polarized pair.
\end{proof}

So for a given polarization type~$T$, once we picked a connected
component~$\cM^0_\marked$, we can distinguish each connected
component~$\cM^0_T$ of~$\cM_T$ by its deformation type~$\tau$. We can thus
write~$\cM_\tau$ instead of~$\cM^0_T$.

A first observation is that, if the group~$\Mon(\Lambda)$ is a proper subgroup
of~$\O(\Lambda)$, an~$\O(\Lambda)$-orbit may contain
several~$\Mon(\Lambda)$-orbits and consequently, the corresponding polarized
moduli space~$\cM_T$ may have several components. As the result of
Apostolov~\cite{Apostolov} shows, this is indeed the case for certain polarization
types of~$\KKK^{[m]}$-type manifolds. We will give a simplified expression for the
exact number of components in Proposition~\ref{prop:numberoftau}.

Finally, we explain the construction of the polarized period map and the
statement of the polarized global Torelli theorem, as mentioned in the
introduction. For a polarization type~$T$, we consider the connected
component~$\cM_T^0=\cM_\tau$ of the polarized moduli space~$\cM_T$
corresponding to a $\Mon(\Lambda)$-orbit~$\tau$ and pick some~$h\in\tau$.  We
consider the stabilizer groups
\[
\O(\Lambda,h)\coloneqq\setmid{\phi\in\O(\Lambda)}{\phi(h)=h}
\quad\text{and}\quad
\Mon(\Lambda,h)\coloneqq\Mon(\Lambda)\cap \O(\Lambda,h).
\]

For a polarized pair~$(X,H)$ of deformation type~$\tau$, if we pick a suitable
marking~$\eta$ in the connected component~$\cM^0_\marked$ such
that~$\eta(H)=h$ then, by the ampleness of the class~$H$, the marked
pair~$(X,\eta)$ must lie in~$\cM^\amp_h$. By quotienting out the action of the
monodromy group, we get the following result \cite[Lemma 8.1, Lemma 8.3, and
Theorem 8.4]{Markman:survey}.

\begin{theorem}[Markman]\leavevmode
\begin{itemize}
\item[(i)] The marked period map
\eqref{eq:marked-period} descends to an open embedding of analytic spaces
\[
\cM_h^\amp/\Mon(\Lambda,h)\into \Omega_h/\Mon(\Lambda,h),
\]
where the second quotient~$\Omega_h/\Mon(\Lambda,h)$ is a normal
quasi-projective variety by Baily--Borel theory. We denote this quotient
by~$\cP_\tau$, since if we choose another~$h'\in\tau$, the two quotients are
canonically isomorphic.
\item[(ii)]
For each~$h\in\tau$, there is an isomorphism of analytic spaces
\[
\cM_\tau\simto\cM^\amp_h/\Mon(\Lambda,h).
\]
The composition with the above embedding gives the polarized period map
\[
\wp_\tau\colon\cM_\tau\into\cP_\tau,
\]
which is an open immersion of algebraic varieties.\qedhere
\end{itemize}
\end{theorem}
Notice that if~$\tau$ and~$\tau'$ are different~$\Mon(\Lambda)$-orbits
contained in~$T$, the quotients~$\cP_\tau$ and~$\cP_{\tau'}$ are isomorphic
but in general not canonically. This can be seen as follows. We consider the
quotient~$(\Omega_h\sqcup\Omega_{-h}) /\O(\Lambda,h)\simeq
\Omega_h/\O^+(\Lambda,h)$, which is again a normal quasi-projective variety.
This quotient can be denoted by~$\cP_T$, since if another~$h'\in T$ is chosen,
the two quotients are canonically isomorphic.  We see that~$\cP_\tau$ is a
covering space of~$\cP_T$ and it admits an action of the
group~$\O^+(\Lambda,h)/ \Mon(\Lambda,h)$, not necessarily free. The
deck transformation group~$G$ will be some quotient of this group. Thus we
have a diagram
\begin{equation}\label{eq:polarized-period}
\begin{tikzcd}[column sep=0]
\wp_\tau\colon\cM_\tau\rar[hook]&[5em]{\cP_\tau}\dar{/G}&=
&\Omega_h/\Mon(\Lambda,h)\\
&{\cP_T}&=&\Omega_h/\O^+(\Lambda,h)
\end{tikzcd}
\end{equation}
In particular, when~$G$ is non-trivial, for two deformation types~$\tau$
and~$\tau'$, there is no canonical isomorphism between the period
domains~$\cP_\tau$ and~$\cP_{\tau'}$: any two such isomorphisms differ by the
action of an element in~$G$ (to be more precise, in this case we have two
groups~$G_\tau$ and~$G_{\tau'}$ that are non-canonically isomorphic).

\begin{remark}
For K3 surfaces, the monodromy group~$\Mon(\Lambda)$ coincides
with~$\O^+(\Lambda)$, and each polarization is characterized by its
square~$2d$. Each period domain~$\cP_T=\cP_{2d}$ is given above as the
quotient~$(\Omega_h\sqcup \Omega_{-h})/\O(\Lambda,h)$. This is usually
formulated in terms of the orthogonal lattice~$h^\perp$: the hyperplane
section~$(\Omega_h\sqcup \Omega_{-h})$ can be identified as the following space
\[
\Omega_{h^\perp}\coloneqq\setmid{[x]\in\bP\big((h^\perp)_\bC\big)}{(x,x)=0,
(x,\bar x)>0},
\]
and by Proposition~\ref{prop:ghs3.4} below, the group~$\O(\Lambda,h)$ restricts
to a subgroup~$\tO(h^\perp)$ of~$\O(h^\perp)$, so~$\cP_{2d}$ can also be given
as the quotient~$\Omega_{h^\perp}/\tO(h^\perp)$.
\end{remark}

\begin{remark}
Another subtlety is that the polarized period map depends on the initial choice
of the connected component~$\cM_\marked^0$ for the definition of deformation
types: if we choose another connected component by acting on the marking using
an element in $\Mon(\Lambda)\cdot\O(\Lambda,h)$, the deformation
type---the~$\Mon(\Lambda)$-orbit---of~$\cM_T^0$ is still~$\tau$, but the period
map is acted on by some element in~$G$; if we choose another connected
component by acting on the marking using an element in the larger
group~$\O(\Lambda)$, the deformation type of~$\cM_T^0$ may change to an
entirely different~$\tau'$, in which case the period map maps the
component~$\cM_T^0$ to a different~$\cP_{\tau'}$ that, as we already stated,
can only be identified with~$\cP_\tau$ up to the action of some element in~$G$.
In Markman's survey, this subtlety is handled by taking disjoint copies
of~$\cM^\amp_h$ (resp.~$\Omega_h$) and by quotienting out by the
action of~$\O(\Lambda)$ to get a canonically defined polarized moduli space
(resp.~polarized period domain). This approach is certainly more canonical as
it does not depend on the particular choice of a connected
component~$\cM^0_\marked$.  However, it is more difficult to describe the
connected components of~$\cM_T$ in this setting.
\end{remark}

Before ending this section, we review some lattice theoretical results that
will be used later.  We first recall some basic definitions. Let~$\Lambda$
be a lattice with isometry group~$\O(\Lambda)$. The
\emph{divisibility}~$\div(x)$ of a primitive element~$x$ in~$\Lambda$ is the
positive generator~$\gamma$ of the subgroup~$(x,\Lambda)$ of~$\bZ$. The
\emph{discriminant group} of~$\Lambda$ is the finite abelian
group~$D(\Lambda)\coloneqq \Lambda^\vee/\Lambda$. We
define~$x_*\coloneqq[x/\div(x)]$, which is an element of~$D(\Lambda)$ of
order~$\div(x)$. When~$\Lambda$ is even, the quadratic form
on~$\Lambda$ induces a~$(\bQ/2\bZ)$-valued quadratic form on~$D(\Lambda)$, and
there is a natural homomorphism~$\chi\colon\O(\Lambda)\to\O(D(\Lambda))$. In
this case, we let~$\tO(\Lambda)$ and~$\hO(\Lambda)$ be the respective preimages
of~$\set1$ and~$\set{\pm1}$ by~$\chi$. We have the following results from
lattice theory.

\begin{proposition}[{\cite[Theorem 1.14.2]{Nikulin}}]\label{prop:chisurj}
For any even indefinite lattice~$\Lambda$ of rank larger than or equal to the
minimal number of generators of~$D(\Lambda)$ plus~$2$, the
homomorphism~$\chi\colon\O(\Lambda)\to\O(D(\Lambda))$ is surjective.
\end{proposition}

\begin{proposition}[{\cite[Lemma 3.2]{GHS:ihs}}]\label{prop:ghs3.4}
Let~$\Lambda$ be an even unimodular lattice and let~$x$ be an
element~of~$\Lambda$ with non-zero square. Denote by~$x^\perp$ the
orthogonal of~$x$ in~$\Lambda$.  We have
\[
\O(\Lambda,x)|_{x^\perp}=\tO(x^\perp),
\]
where~$\O(\Lambda,x)$ is the stabilizer group of~$x$ in~$\O(\Lambda)$.
\end{proposition}

\begin{proposition}[Eichler's criterion, {\cite[Lemma 3.5]{GHS:ihs}}]
\label{prop:eichler}
Let~$\Lambda$ be an even lattice which contains at least two orthogonal copies
of the hyperbolic plane~$U$. The~$\tO(\Lambda)$-orbit of a primitive
element~$x$ is determined by its square~$x^2$ and the class $x_*=[x/\div(x)]$
in~$D(\Lambda)$.
\end{proposition}
The Eichler's criterion can be slightly strengthened by replacing $\tO(\Lambda)$
with smaller subgroups.
\begin{proposition}
\label{prop:eichler_SO}
Under the same assumption for $\Lambda$ as above,
for a primitive element~$x$, the following three orbits coincide
\[
\tO(\Lambda)x = \tSO(\Lambda)x = \tSO^+(\Lambda)x.
\]
In particular, all three orbits are determined by the square $x^2$ and the class
$x_*$ in $D(\Lambda)$.
\end{proposition}
\begin{proof}
Write $\Lambda=U_1\oplus U_2\oplus \Lambda_0$ where $U_1$ and $U_2$ are two
copies of the hyperbolic plane $U$. Since $U$ is unimodular, by Eichler's
criterion, we may find $\phi\in \tO(\Lambda)$ such that $\phi(x)\in U_2\oplus
\Lambda_0$. Take $u,v\in U_1$ with $u^2=2$ and $v^2=-2$, then the
reflections $R_u,R_v$ lie in $\O(\Lambda,\phi(x))$ and they satisfy
$\sigma(R_u)=-1$, $\sigma(R_v)=1$, $\chi(R_u)=\chi(R_v)=1$, and
$\det(R_u)=\det(R_v)=-1$.

Now for $\varphi\in \tO(\Lambda)$ with $\det(\varphi)=-1$, we have
$\varphi(x)=\varphi\circ\phi^{-1}\circ\phi(x)
=\varphi\circ\phi^{-1}\circ R_u\circ \phi(x)$, and the element
$\varphi\circ\phi^{-1}\circ R_u\circ \phi$ has determinant 1, so $\varphi(x)$
lies in the same $\tSO(\Lambda)$-orbit as $x$ and we get $\tO(\Lambda)x =
\tSO(\Lambda)x$.

Similarly, for $\varphi \in \tSO(\Lambda)$ with $\sigma(\varphi)=-1$,
we have $\varphi(x)=\varphi\circ\phi^{-1}\circ\phi(x)
=\varphi\circ\phi^{-1}\circ R_u\circ R_v\circ \phi(x)$, and the element
$\varphi\circ\phi^{-1}\circ R_u\circ R_v\circ \phi$ lies in $\tSO^+(\Lambda)$,
so we get $\tSO(\Lambda)x = \tSO^+(\Lambda)x$.
\end{proof}

\section{Monodromy group and number of components}\label{sec:number}
In this section, we will calculate the number of components of the moduli
space~$\cM_T$ of a given polarization type~$T$, for all known deformation
types.
The polarization type determines the square and the
divisibility of its elements, but the converse is in general not true: we will
calculate the number of~$T$ with given square and divisibility.

First we recollect the descriptions for the lattice $\Lambda=H^2(X,\bZ)$ and
the monodromy group $\Mon(\Lambda)$ for all known deformation types. The
lattice structures for $\KKK^{[m]}$ and $\Kum_m$ are known by Beauville
\cite{Beauville}, and for $\OG_6$ and $\OG_{10}$ they are computed by
Rapagnetta \cite{Rapagnetta}. The monodromy group is computed by Markman in the
$\KKK^{[m]}$-case, Markman and Mongardi in the $\Kum_m$-case
\cite{Markman:Kummer, Mongardi:mon},
Mongardi--Rapagnetta for $\OG_6$ \cite{MongardiRapagnetta}, and Onorati for
$\OG_{10}$ \cite{Onorati:OG10}.
\begin{theorem}
The descriptions for the lattice $\Lambda=H^2(X,\bZ)$ and the monodromy
group $\Mon(\Lambda)$ for all known deformation types are as follows.
\begin{table}[H]
\[
\begin{array}{|c|c|c|c|}
\hline
&\Lambda=H^2(X,\bZ)&D(\Lambda)&\Mon(\Lambda)\\
\hline
\KKK&U^{\oplus3}\oplus E_8(-1)^{\oplus2}&0&\O^+(\Lambda)\\
\KKK^{[m]}&\Lambda_\KKK\oplus\inner{-(2m-2)}&\bZ/(2m-2)\bZ&\hO^+(\Lambda)\\
\Kum_m&U^{\oplus3}\oplus\inner{-(2m+2)}&\bZ/(2m+2)\bZ&\setmid{g\in
\hO^+(\Lambda)}{\chi(g)\cdot\det(g)=1}\\
\OG_6&U^{\oplus3}\oplus\inner{-2}^{\oplus2}&(\bZ/2\bZ)^2&\O^+(\Lambda)\\
\OG_{10}&\Lambda_\KKK\oplus\begin{psmallmatrix}-6&3\\3&-2\end{psmallmatrix}&\bZ/3\bZ&\O^+(\Lambda)\\
\hline
\end{array}
\]
\caption{Lattice and monodromy group for known deformation types}
\label{table:mon}
\end{table}

\noindent
Here~$U$ is the hyperbolic plane,~$E_8(-1)$ is the~$E_8$-lattice with negative
definite form, and~$\inner{k}$ is the lattice generated by one primitive
element with square $k$.
\end{theorem}

We may compute the number of components for a given polarization type~$T$ using
Proposition~\ref{prop:deftype}. We first prove a lemma concerning the
orthogonal group of the discriminant group~$D(\Lambda)$.

\begin{lemma}
\label{lemma:OG}
Let~$D$ be a cyclic group of order~$2n$ with a quadratic form $q\colon
D\to\bQ/2\bZ$. If there is a generator~$g\in D$ with~$q(g)=\frac{1}{2n}$, then
\[
\O(D)=\setmid{g\mapsto ag}
{\begin{gathered}
a\in\bZ/2n\bZ\\
a^2\equiv1\pmod{4n}
\end{gathered}}\simeq(\bZ/2\bZ)^{\rho(n)},
\]
where~$\rho(n)$ denotes the number of distinct prime divisors of~$n$.
\end{lemma}
\begin{proof}
Write $n=p_1^{\alpha_1}\cdots p_r^{\alpha_r}$ with $r=\rho(n)$. If~$n$ is
odd,~$a$ is determined by the conditions~$a\equiv 1\pmod 2$ and $a\equiv\pm
1\pmod{p_i^{\alpha_i}}$; if~$n$ is even, we let $p_1=2$, then~$a$ is determined
by the conditions $a\equiv \pm1\pmod{2^{\alpha_1+1}}$ and $a\equiv\pm
1\pmod{p_i^{\alpha_i}}$ for~$i\ge2$. In both cases, we
have~$\O(D)\simeq(\bZ/2\bZ)^{\rho(n)}$.
\end{proof}
The Eichler's criterion allows us to compute the number of
$\tO(\Lambda)$-orbits.
\begin{lemma}
\label{lemma:number_tO_orbits}
Let $\Lambda$ be an even lattice containing at least two orthogonal copies of
the hyperbolic plane $U$, such that the discriminant group $D(\Lambda)$ is
cyclic of order $2n$. Then for each $\O(\Lambda)$-orbit $T$ of a primitive
element with divisibility $\gamma$, the number of $\tO(\Lambda)$-orbits
contained in $T$ is equal to $2^{\tilde{\rho}(\gamma)}$,
where~$\tilde{\rho}(n)$ is equal to~$\rho(n)$---the number of distinct prime
divisors of~$n$---if~$n$ is odd, and~$\rho(n/2)$ if~$n$ is even.
\end{lemma}
\begin{proof}
Fix one element~$h\in T$ so that~$T$ is the set~$\setmid{\phi(h)}{\phi\in\O(\Lambda)}$.
By Eichler's criterion (Proposition~\ref{prop:eichler}), as the square is
fixed, the number of $\tO(\Lambda)$-orbits is the same as the number of
possible values of~$\big(\phi(h)\big)_*=\chi(\phi)(h_*)\in
D(\Lambda)$ for all~$\phi\in\O(\Lambda)$. The lattice~$\Lambda$ satisfies the
condition in
Proposition~\ref{prop:chisurj}, so the
homomorphism~$\chi\colon\O(\Lambda)\to\O(D(\Lambda))$ is surjective.
Therefore it suffices to count the number of possible~$a h_*\in D(\Lambda)$
for all~$a\in \O(D(\Lambda))$. Since~$h$ is primitive of divisibility~$\gamma$,
the class~$h_*=[h/\gamma]$ is of order~$\gamma$. Viewing the isometry~$a$ as an
element of~$\bZ/2n\bZ$, we therefore need to count the number of
possible remainders of~$a$ modulo~$\gamma$ under the quotient
map~$\bZ/2n\bZ\to\bZ/\gamma\bZ$.

Using a similar argument as in the proof of Lemma~\ref{lemma:OG}, we write
$\gamma=p_1^{\alpha_1}\cdots p_r^{\alpha_r}$ with $r=\rho(\gamma)$. If~$\gamma$
is odd, then~$a$ modulo~$\gamma$ can take all the values satisfying~$a\equiv\pm
1\pmod{p_i^{\alpha_i}}$. If $\gamma$ is even, let $p_1=2$; if $\gamma$ is not
divisible by~$4$, that is, $\alpha_1=1$, then~$a$ modulo~$\gamma$ can take all the
values satisfying $a\equiv 1\pmod 2$ and $a\equiv\pm 1\pmod{p_i^{\alpha_i}}$; if
$\alpha_1\ge2$, $a$ modulo~$\gamma$ can take all the values satisfying~$a\equiv
\pm1\pmod{2^{\alpha_1+1}}$ and $a\equiv\pm 1\pmod{p_i^{\alpha_i}}$ for $i\ge2$.
Combining all three cases, the number of~$\tO(\Lambda)$-orbits
is equal to~$2^{\tilde{\rho}(\gamma)}$.
\end{proof}

Now we can compute the number of connected components.

\begin{proposition}
\label{prop:numberoftau}
Let $X$ be a hyperkähler manifold and $T$ be a polarization type of
divisibility~$\gamma$ on $X$.
\begin{itemize}
\item If $X$ is of $\KKK^{[m]}$-type or $\Kum_m$-type, the number of connected
components of the polarized moduli space~$\cM_T$ is equal
to~$2^{\max(\tilde{\rho}(\gamma)-1,0)}$.
\item If $X$ is of $\OG_6$-type or $\OG_{10}$-type, the polarized moduli space
$\cM_T$ is connected.\qedhere
\end{itemize}
\end{proposition}
\begin{proof}
As Proposition~\ref{prop:deftype} shows, the number of connected components of
$\cM_T$ is equal to the number of~$\Mon(\Lambda)$-orbits contained in
the~$\O(\Lambda)$-orbit~$T$.
We fix one element~$h\in T$, so~$T$ is the
set~$\setmid{\phi(h)}{\phi\in\O(\Lambda)}$.

\bigskip

\noindent\textbf{Case $\KKK^{[m]}$:}
The discriminant group $D(\Lambda)$ is cyclic of order $2m-2$, so
Lemma~\ref{lemma:number_tO_orbits} applies and the number of
$\tO(\Lambda)$-orbits contained in $T$ is equal to $2^{\tilde{\rho}(\gamma)}$.

Since the subgroup~$\hO(\Lambda)$ is generated by~$\tO(\Lambda)$
and~$-\Id$, we see that if~$h$ and~$-h$ are in the same~$\tO(\Lambda)$-orbit,
that is, when~$\gamma$ is~$1$ or~$2$, the number of~$\hO(\Lambda)$-orbits is the
same as the number of~$\tO(\Lambda)$-orbits; otherwise it should be divided
by~$2$. So this gives $2^{\max(\tilde{\rho}(\gamma)-1,0)}$ as the number
of~$\hO(\Lambda)$-orbits.

To conclude, we show that the~$\hO(\Lambda)$-orbits and the~$\hO^+(\Lambda)$-orbits
are the same.
Following the proof of Proposition~\ref{prop:eichler_SO}, there is an element
$R\in \O(\Lambda,h)$ (namely $R_u$) with $\sigma(R)=-1$ and $\chi(R)=1$.
Now for $\phi\in\hO(\Lambda)$ with $\sigma(\phi)=-1$, we have
$\phi(h)=\phi\circ R(h)$, where $\phi\circ R$ lies in $\hO^+(\Lambda)$.
So $\phi(h)$ lies in the same $\hO^+(\Lambda)$-orbit as $h$ and therefore
$\hO(\Lambda)h = \hO^+(\Lambda)h$.

\bigskip

\noindent \textbf{Case $\Kum_m$:}
The discriminant group $D(\Lambda)$ is cyclic of order $2m+2$, so
again Lemma~\ref{lemma:number_tO_orbits} applies and we get
$2^{\tilde{\rho}(\gamma)}$ as the number of $\tO(\Lambda)$-orbits. By
Proposition~\ref{prop:eichler_SO}, this is also the number of
$\tSO^+(\Lambda)$-orbits.

Moreover, following the proof of Proposition~\ref{prop:eichler_SO}, there
exists an element $R\in O(\Lambda,h)$ (namely $R_u\circ R_v$) such that
$\sigma(R)=-1$, $\chi(R)=1$, $\det(R)=1$. On the other hand,
we note that $\sigma(-\Id)=-1$, $\chi(-\Id)=-1$, $\det(-\Id)=-1$. This shows that
$\Mon(\Lambda)$ is generated by $\tSO^+(\Lambda)$ and $-R$.
If $h$ and $-h=-R(h)$ are in the same $\tSO^+(\Lambda)$-orbit, that is, when $\gamma$
is 1 or 2, then the number of $\Mon(\Lambda)$-orbits is the same as the number
of $\tSO^+(\Lambda)$-orbits; otherwise it should be divided by 2.
So again we obtain $2^{\max(\tilde{\rho}(\gamma)-1,0)}$ as the number of
$\Mon(\Lambda)$-orbits.

\bigskip

\noindent \textbf{Case $\OG_6$ and $\OG_{10}$:}
In these two cases, the monodromy group is equal to $\O^+(\Lambda)$.
Again, following the proof of Proposition~\ref{prop:eichler_SO}, there
exists a reflection $R\in O(\Lambda,h)$ (namely $R_u$) such that
$\sigma(R)=-1$. So one may conclude that the $\O^+(\Lambda)$-orbit of $h$
coincides with the entire $\O(\Lambda)$-orbit $T$.
\end{proof}

We also have the following result on the number of polarization types with
given square and divisibility on a hyperkähler manifold of $\KKK^{[m]}$-type or
$\Kum_m$-type. Together with Proposition~\ref{prop:numberoftau}, this
gives a refined version of Apostolov's \cite{Apostolov} result for $\KKK^{[m]}$ and
Onorati's \cite{Onorati:kummer} result for $\Kum_m$ (\cf also
\cite[Proposition 3.6]{GHS:ihs}).
\begin{proposition}\label{prop:numberofT}
Let~$m$,~$n$, and~$\gamma$ be positive integers with~$m\ge2$.
Let $\tilde m$ be $m-1$ for the $\KKK^{[m]}$-type and $m+1$ for the $\Kum_m$-type, so
in both cases we have $D(\Lambda)\simeq \bZ/2\tilde m\bZ$.
Moreover we assume that
$\gamma\mid\operatorname{gcd}(2\tilde m,2n)$. For a prime divisor~$p$
of~$\gamma$, set $\alpha\coloneqq\min(v_p(\tilde m),v_p(n))$ and
$\beta\coloneqq v_p(\gamma)$, where $v_p$ is the $p$-adic valuation.
Then there exists a polarization type~$T$ of square~$2n$ and of
divisibility~$\gamma$, if and only if
the following conditions are satisfied for all prime divisors~$p$ of~$\gamma$:
\begin{itemize}
\item if $v_p(\tilde m)\ne v_p(n)$, then $\beta\le \alpha/2$;
\item if $v_p(\tilde m)=v_p(n)=\alpha$, then either $\beta\le\alpha/2$, or
$\beta>\alpha/2$ and~$-n/\tilde m$ is a square modulo~$p^{2\beta-\alpha}$.
\end{itemize}
The total number of these~$T$ is given by the product $\prod_{p\mid\gamma}N_p$,
where for~$p\ge3$
\[
N_p\coloneqq\begin{cases}\frac12(p-1)p^{\beta-1}&\textnormal{if }
\beta\le\alpha/2;\\p^{\alpha-\beta}&\textnormal{if }\beta>\alpha/2;\end{cases}
\]
and for~$p=2$
\[
N_2\coloneqq \begin{cases}1&\textnormal{if }\beta= 1;\\
2^{\beta-2}&\textnormal{if }\beta\ge2,\beta\le\alpha/2+1;\\
2^{\alpha+1-\beta}&\textnormal{if }\beta>\alpha/2+1.\end{cases}
\]
\end{proposition}

\begin{proof}
For the $\KKK^{[m]}$-type and the $\Kum_m$-type, we have
$\Lambda=\Lambda_0\oplus \bZ\delta$, where $\Lambda_0$ is an even
unimodular lattice containing three orthogonal copies of the hyperbolic plane
$U$, and $\delta$ is of square $-2\tilde m$. The discriminant group is cyclic
of order $2\tilde m$, generated by $\delta_*$.

We first study the existence of a polarization type with given square and
divisibility.  Let~$h\in\Lambda$ be a primitive element of
divisibility~$\gamma$.  If~$\gamma=1$, since~$\Lambda_0$ contains orthogonal
copies of $U$, it is clear that a polarization type of
square $2n$ exists for all $n>0$. So we will look at~$\gamma\ge2$. We write
\[
h=\gamma ax+b\delta,
\]
where $x\in\Lambda_0$ is primitive of square~$x^2=2c$, with $a,b,c\in\bZ$
such that $\gcd(\gamma a,2\tilde m)=\gamma$ and $\gcd(\gamma a,b)=1$. Suppose
that~$h$ is of square~$2n$. We obtain the relation
\[
2n=h^2=2ca^2\gamma^2-b^2\cdot 2\tilde m.
\]
For such an~$h$ to exist, it is necessary and sufficient that there exist
some integer~$b$ satisfying
\[
\gamma^2\mid b^2\tilde m+n.
\]
For each prime divisor $p$ of~$\gamma$, since~$\gcd(\gamma,b)=1$, we see
that~$b$ is not divisible by~$p$. So if $v_p(\tilde m)\ne v_p(n)$, then
$v_p(b^2\tilde m+n)=\min(v_p(\tilde m),v_p(n))$ and we obtain the first
condition; if $v_p(\tilde m)=v_p(n)=\alpha$, then for $p^{2\beta}\mid b^2\tilde
m+n$ to hold we obtain the second condition.

Given the square and the divisibility, to count the number
of such~$\O(\Lambda)$-orbits~$T$, we first count the number
of~$\tO(\Lambda)$-orbits.  Any such element~$h$ can again be expressed
as~$\gamma a x+b\delta$. By Eichler's criterion, since the square is fixed, the
number of~$\tO(\Lambda)$-orbits is just the number of possible~$h_*=
\frac{b\cdot 2\tilde m}{\gamma}\delta_*$, or equivalently, the number of
possible remainders of~$b$ modulo~$\gamma$. We thus express this number as the
product~$\prod_{p\mid\gamma}M_p$, where~$M_p$ is the number of possible
remainders of~$b$ modulo~$p^\beta$.

For~$p\ge3$, if $\beta\le\alpha/2$, then we only need $\gcd(b,p)=1$, thus~$M_p$
is equal to~$(p-1)p^{\beta-1}$; if~$\beta>\alpha/2$, then the equation
$b^2\equiv-n/\tilde m\pmod{p^{2\beta-\alpha}}$ has two solutions,
thus~$M_p$ is equal to~$2p^{\alpha-\beta}$.

For~$p=2$, as $\gcd(b,p)=1$, we see first that~$b$ is necessarily odd.
If~$\beta\le\alpha/2+1$, we will show that this is also sufficient,
so~$M_2$ is equal to~$2^{\beta-1}$. To prove this, we distinguish three cases:
if~$\beta\le\alpha/2$, it is clear that~$b^2\tilde m+n$ is a multiple of
$2^{2\beta}$; if $\beta=\alpha/2+1/2$, then $v_p(\tilde m)=v_p(n)=\alpha$ and
$b^2\tilde m+n$ is a multiple of $2^{\alpha+1}=2^{2\beta}$; if
$\beta=\alpha/2+1$, then $v_p(\tilde m)=v_p(n)=\alpha$ and $-n/\tilde
m\equiv1\pmod4$, so $b^2\tilde m+n$ is a
multiple of $2^{\alpha+2}=2^{2\beta}$. If $\beta>\alpha/2+1$, the equation
$b^2\equiv1\pmod{2^{2\beta-\alpha}}$ has two solutions
modulo~$2^{2\beta-\alpha-1}$, so $M_2$ is equal to~$2\times2^{\alpha+1-\beta}$.

To conclude, as Lemma~\ref{lemma:number_tO_orbits} shows that
each~$\O(\Lambda)$-orbit~$T$ contains~$2^{\tilde\rho(\gamma)}$
different~$\tO(\Lambda)$-orbits, the number
of~$T$ is given by $\prod_{p\mid\gamma}M_P$ divided
by~$2^{\tilde\rho(\gamma)}$. We let
$N_p=M_p/2$ for~$p\ge3$, $N_2=M_2/2$ if $v_2(\gamma)\ge2$, and $N_2=M_2=1$ if
$v_2(\gamma)=1$. This gives the desired formula.
\end{proof}

For completeness, we also provide the results for $\OG_6$ and $\OG_{10}$.

\begin{proposition}
Let $n$ and $\gamma$ be positive integers.
For the $\OG_6$-type and the $\OG_{10}$-type, a polarization type $T$ is
uniquely determined by its square $2n$ and divisibility $\gamma$.
\begin{itemize}
\item For the $\OG_6$-type, such $T$ exists if and only if $\gamma=1$, or
$\gamma=2$ and $n\equiv 2,3\pmod 4$;
\item for the $\OG_{10}$-type, such $T$ exists if and only if $\gamma=1$, or
$\gamma=3$ and $n\equiv 6\pmod 9$.\qedhere
\end{itemize}
\end{proposition}
\begin{proof}
In both cases, since the lattice $\Lambda$ contains orthogonal copies of
$U$, the existence of a polarization type of square $2n$ and divisibility 1 is
clear, and the uniqueness follows from Eichler's criterion.

For the $\OG_6$-type, we write $u$ and $v$ for the two generators with square
$-2$ so $\Lambda=\Lambda_0\oplus\bZ u\oplus \bZ v$. Each primitive element $h$
of divisibility 2 can be written as
\[
h=2 ax+bu+cv,
\]
where $x\in\Lambda_0$ is primitive with $x^2=2d$ and $a,b,c,d\in \bZ$,
such that $\gcd(2a,b,c)=1$. In particular, $b$ and $c$ cannot be both even,
and the class $h_*$ is given by $(\bar b,\bar c)\in(\bZ/2\bZ)^2$.
Suppose that~$h$ is of square~$2n$. We obtain the relation
\[
2n=h^2=8a^2d-2b^2-2c^2,
\]
and we may deduce that $4\mid n+b^2+c^2$. If $n\not\equiv
2,3\pmod4$ there are no integer solutions. If $n\equiv 2\pmod 4$, then $b$ and
$c$ must both be odd, so $h_*=(\bar 1,\bar 1)$ and by Eichler's criterion all
such $h$ lie in the same $\tO(\Lambda)$-orbit, so the $\O(\Lambda)$-orbit is
also unique. If $n\equiv 3\pmod 4$, then $b$ and $c$ must be one odd one even,
so $h_*$ can either be $(\bar 1,\bar0)$ or $(\bar 0,\bar 1)$, and by Eichler's
criterion there are two $\tO(\Lambda)$-orbits, but the map that interchanges
the coordinates $u$ and $v$ is an isometry, so these two lie in the same
$\O(\Lambda)$-orbit, and again we get the uniqueness.

For the $\OG_{10}$-type, we similarly write $u$ and $v$ for the two generators
with matrix $\begin{psmallmatrix}-6&3\\3&-2\end{psmallmatrix}$, so
$\Lambda=\Lambda_0\oplus\bZ u\oplus \bZ v$. Each primitive element $h$ of
divisibility 3 can be written as
\[
h=3 ax+bu+3cv,
\]
where $x\in\Lambda_0$ is primitive with $x^2=2d$ and $a,b,c,d\in \bZ$,
such that $\gcd(3a,b,3c)=1$. In particular, $b$ is not divisible by 3,
and the class $h_*$ is given by $\bar b\in \bZ/3\bZ$.
Suppose that~$h$ is of square~$2n$. We obtain the relation
\[
2n=h^2=18a^d-6b^2+18bc-18c^2,
\]
and we may deduce that $9\mid n+3b^2$, so we must have $n\equiv 6\pmod 9$.
By Eichler's criterion there are two $\tO(\Lambda)$-orbits depending on the
value $h_*\in D(\Lambda)=\bZ/3\bZ$, but $-\Id$ interchanges the two non-zero
classes in $D(\Lambda)$ so again the two lie in the same $\O(\Lambda)$-orbit.
\end{proof}

\section{Image of the period map}\label{sec:image}
We will now study the image of the polarized period map.
For all known deformation types, the complement of the image in the period domain
can be shown to be a finite union of divisors: we will give explicit numerical
conditions describing these divisors.
The image of the period map is closely related to the determination of the
ample cone, which has been settled for all known deformation types, so we first
review the results.

Recall that on $H^{1,1}(X,\bR)$ the Beauville--Bogomolov--Fujiki form induces a
quadratic form of signature $(1,b_2-3)$, so the cone of positive classes has
two connected components, and we call the one containing a Kähler class the
\emph{positive cone} and denote it by $\cC_X$. The cone of all Kähler classes
sits inside $\cC_X$ and is denoted by $\cK_X$. We also consider the
\emph{birational Kähler cone} $\cB\cK_X$, which is the union $\bigcup
f^{-1}\cK_{X'}$ over all birational maps $f$ from $X$ to some other hyperkähler
manifold $X'$.
The Néron--Severi group~$\NS(X)$ is a sublattice~$H^2(X,\bZ)\cap
H^{1,1}(X,\bR)$ inside~$H^2(X,\bZ)$.

We have the following crucial notion: a divisor $D$ on $X$ is called a
\emph{wall divisor}, if $D^2<0$ and $f(D^\perp)\cap \cB\cK_X=\emptyset$ for
all monodromy operators $f$ (\cf \cite[Definition~1.2]{Mongardi:note} and
\cite[Definition~1.13]{AmerikVerbitsky}).
The property of being a wall divisor is stable under parallel transport
operators \cite[Theorem~3.1]{Mongardi:note}.
\begin{theorem}[Mongardi]
Let $(X,\eta)$ and $(X',\eta')$ be two marked hyperkähler manifolds lying in
the same connected component $\cM^0_\marked$ of the marked moduli space. Let
$D\in\NS(X)$ and $D'\in \NS(X')$ be divisors such that $\eta^{-1}\circ
\eta(D)=D'$. Then $D$ is a wall divisor on $X$ if and only if $D'$ is a wall
divisor on $X'$.
\end{theorem}
Once we picked a connected component $\cM^0_\marked$, we may extend this notion
to elements of the lattice $\Lambda$: a class $\kappa\in\Lambda$ with
$\kappa^2<0$ is called a \emph{wall class}, if for all $(X,\eta)\in
\cM^0_\marked$ such that the class $\eta^{-1}(\kappa)$ is of type $(1,1)$, it
gives a wall divisor on $X$.
Clearly the property only depends on the $\Mon(\Lambda)$-orbit of $\kappa$.
Wall divisors give a chamber decomposition on the positive cone $\cC_X$, and
the Kähler cone $\cK_X$ is given by one of the chambers.

For $\KKK^{[m]}$-type and $\Kum_m$-type, a numerical characterization for wall
divisors is known.
Write as before $\tilde m=m-1$ for $\KKK^{[m]}$-type and $\tilde m=m+1$ for $\Kum_m$-type.
Recall that in these two cases, the lattice $\Lambda$ has the form
$\Lambda=\Lambda_0\oplus\bZ\delta$,
where $\Lambda_0$ is an even unimodular lattice containing three orthogonal
copies of $U$, and $\delta$ is of square $-2\tilde m$.
We also consider the \emph{Mukai lattice}
\[
\tL\coloneqq \Lambda_0\oplus U,
\]
which is even and unimodular.
For any vector~$v\in\tL$ of square~$2\tilde m$, the sublattice~$v^\perp$ is
isomorphic to $\Lambda$. Since all such~$v$ are in the same~$\O(\tL)$-orbit
due to the unimodularity of~$\tL$, we may fix $v=u_1+\tilde mu_2$, where
$\inner{u_1,u_2}$ is a copy of the hyperbolic plane~$U$, and identify~$\Lambda$ as
the sublattice~$v^\perp$. In particular we set~$\delta=u_1-\tilde mu_2$.

When $X$ is of $\KKK^{[m]}$-type or $\Kum_m$-type, there is an embedding
of~$H^2(X,\bZ)$ into~$\tL$, canonical up to the action of~$\O(\tL)$
(see \cite[Corollary 9.5]{Markman:survey} for $\KKK^{[m]}$-type, and
\cite[Theorem~4.9]{Wieneck} for $\Kum_m$-type). For any such
embedding, the orthogonal of its image is generated by a vector of
square~$2\tilde m$. So we can assume that the image is exactly~$\Lambda$, by
mapping one of these generators to the fixed~$v$ using some element
in~$\O(\tL)$.  In this way, we get a distinguished marking~$\eta\colon
H^2(X,\bZ)\simto \Lambda$, canonical up to the action
of~$\set{\pm\Id}\cdot\O(\tL,v)|_{\Lambda}$. By Proposition~\ref{prop:ghs3.4}, this
group is equal to~$\set{\pm\Id}\cdot \tO(\Lambda)=\hO(\Lambda)$. Therefore we
get the following result.
\begin{proposition}\label{prop:mukai}
Let~$X$ be a hyperkähler manifold of~$\KKK^{[m]}$-type or $\Kum_m$-type.  There is a
distinguished marking
\[
\eta\colon H^2(X,\bZ)\simto\Lambda\subset\tL,
\]
canonical up to the action of~$\hO(\Lambda)$. It induces an
isometry between the two discriminant groups~$D(H^2(X,\bZ))$
and~$D(\Lambda)\simeq\bZ/2\tilde m\bZ$, canonical up to a sign. In other words,
there is a canonical choice of a pair of generators~$\pm g$
for~$D(H^2(X,\bZ))$, mapped to~$\pm\delta_*$ under the isometry.
\end{proposition}
Any monodromy operator must respect the choice of the pair of generators~$\pm g$,
so the monodromy group~$\Mon(\Lambda)$ must lie in the subgroup~$\hO(\Lambda)$,
which is indeed the case.

We now give the description of the Kähler cone $\cK_X$ for these two cases.
The $\KKK^{[m]}$-case is due to the results of Bayer--Macrì,
Bayer--Hassett--Tschinkel, and Mongardi (note that in
\cite{BayerHassettTschinkel}, the
manifold $X$ is assumed to be projective; this assumption can be removed using
\cite{Mongardi:note} or \cite[Theorem~1.17 and 1.19]{AmerikVerbitsky}).
The $\Kum_m$-case is due to Yoshioka \cite{Yoshioka} (see also
\cite{Mongardi:mon}).
\begin{theorem}[Bayer--Macrì, Bayer--Hassett--Tschinkel, Mongardi;
Yoshioka]
\label{thm:BHT}
Let~$X$ be a hyperkähler manifold of~$\KKK^{[m]}$-type or $\Kum_m$-type. Under the
embedding
\[
\eta\colon H^2(X,\bZ)\simto\Lambda\into\tL,
\]
we denote by~$\tL_{\alg}$
the saturation of~$\eta(\NS(X))\oplus\bZ v$.
Consider the set
\[
S\coloneqq
\begin{cases}
\setmid{s\in\tL}{s^2\ge-2,\ |(s,v)|\le \tilde m=m-1}\setminus\set{0}&\text{if $X$ is of
$\KKK^{[m]}$-type;}\\
\setmid{s\in\tL}{s^2\ge0,\ 0<|(s,v)|\le \tilde m=m+1}&\text{if $X$ is of
$\Kum_m$-type.}
\end{cases}
\]
Then the Kähler cone $\cK_X$ is one of the connected components of the positive
cone $\cC_X$ cut out by the hyperplanes~$s^\perp$ in~$\NS(X)_\bR$, for
all~$s\in S\cap \tL_\alg$.
\end{theorem}

Note that the particular choice of the embedding~$\eta$ does not matter here:
because~$\eta$ is unique up to the action of~$\O(\tL)$, and the set~$S$ is
clearly~$\O(\tL)$-invariant.

This description depends on the larger lattice~$\tL$, which is inconvenient to
work with. Note that each~$s\in S$ together with $v$ span a rank-2 sublattice
of $\tL$, so we may consider its intersection with $\Lambda$, which is of rank
1, and pick a generator $\kappa\in\Lambda$. The hyperplane $s^\perp$
can then also be expressed as $\kappa^\perp$. Since the class $\kappa$ lies in
$\NS(X)$ if and only if $s$ lies in $\tL_\alg$, we may conclude that
all wall classes arise this way from some $s\in S$. We now give a
lattice theoretical result, which will yield a numerical criterion for wall
classes $\kappa\in \Lambda$ that is intrinsic to the smaller lattice~$\Lambda$.
\begin{proposition}\label{prop:kappa}
Let $\tL$ be a lattice of the form $\Lambda_0\oplus U$, where $\Lambda_0$ is an
even unimodular lattice and $U$ is the hyperbolic plane with basis $u_1,u_2$.
Let $v = u_1+\tilde m u_2$ and $\delta=u_1-\tilde m u_2$, and let $\Lambda$ be the
sublattice $v^\perp=\Lambda_0\oplus\bZ\delta$.
Let $\kappa\in \Lambda$ be a primitive vector and write~$\kappa^2=2l$ and
$\kappa_*=k\delta_*\in D(\Lambda)\simeq \bZ/2\tilde m\bZ$, where $|k|\le \tilde
m$. Set~$d\coloneqq\gcd(2\tilde m,k)$.
\begin{enumerate}
\item[(i)] There is a unique integer $c$ such that
\[
l=c\left(\frac{2\tilde m}d\right)^2-\tilde m\left(\frac kd\right)^2.
\]
\item[(ii)] Let $a\in\bZ_{\ge 0}$ be a non-negative integer.
There is a non-zero element $s\in \tL$ contained in the saturation of the
sublattice generated by $\kappa$ and $v$, such that
\[
s^2\ge -2a,\quad |(s,v)|\le \tilde m,
\]
if and only if the integer~$c$ in (i) satisfies $c\ge -a$. When this
is the case, there is one such element~$s$ with~$s^2=2c$
and~$(s,v)=-k$.\qedhere
\end{enumerate}
\end{proposition}
\begin{proof}
First we may assume that~$k\ge0$ by changing $\kappa$ to $-\kappa$ if needed.
Since~$\kappa_*=[\kappa/\div(\kappa)]$ is equal to~$k\delta_*=[k\delta/2\tilde
m]$
in~$D(\Lambda)$, we may write
\[
\frac{\kappa}{\div(\kappa)}=x+b\delta+\frac{k\delta}{2\tilde m},
\]
where $x\in\Lambda_0$ and~$b\in\bZ$. Since~$\kappa$ is integral and
primitive, we see that~$\div(\kappa)=\frac{2\tilde m}{d}$.
Now we let
\[
s\coloneqq\frac{d\kappa-kv}{2\tilde m}=x+b\delta-ku_2,
\]
which is an integral class in~$\tL\setminus\set{0}$, with $|(s,v)|=|{-k}|\le
\tilde m$.  Let $s^2=2c$. We can easily verify that $c$ is the integer satisfying
the equality in (i). Moreover, if $c\ge-a$, the vector~$s$ provides the element
we need in (ii).

Conversely, suppose that there is some other vector~$s'$ satisfying
the condition in (ii), then we will show that~$c\ge-a$ so the vector~$s$ itself
satisfies the condition. We let ${s'}^2=2c'$ with $c'\ge-a$, and $(s',v)=-k'$
with $|k'|\le \tilde m$. Since $2\tilde ms'$ lies in the direct sum~$\bZ\kappa
\oplus\bZ v$, there exists a unique integer~$d'$ such that
\[
2\tilde ms'=d'\kappa-k'v \quad\text{or equivalently}\quad
s'=\frac{d'\kappa-k'v}{2\tilde m}.
\]
As~$\kappa$ is of divisibility $\frac{2\tilde m}{d}$ in $\Lambda$, there is some
$y\in \Lambda$ such that $(\kappa,y)=\frac{2\tilde m}{d}$. We then have
$(s',y)=\frac{d'}{d}$, so~$d$ divides~$d'$. Set $d'=\lambda d$. We must
have~$\lambda\ne0$: otherwise, $s'$ is equal to~$-\frac{k'}{2\tilde m}v$;
but~$|k'|\le \tilde m$, so~$s'$ can only be~$0$, which contradicts our
hypothesis. By changing $s'$ to $-s'$ if needed, we may suppose
that~$\lambda\ge 1$. Then by looking at the integral class~$s'-\lambda s$, we
have~$k'\equiv\lambda k\pmod{2\tilde m}$. Write~$k'=\lambda k-\mu\cdot 2\tilde
m$. Since~$k'\le \tilde m$, we must have $\mu\ge0$.  Then we have $s'=\lambda
s+\mu v$ and thus
\begin{align*}
{s'}^2&=\lambda^2s^2+2\lambda\mu(s,v)+\mu^2v^2\\
&=\lambda^2s^2-\mu(2\lambda k-\mu\cdot2\tilde m))\\
&=\lambda^2s^2-\mu(2k'+\mu\cdot2\tilde m)\le\lambda^2s^2,
\end{align*}
where the last inequality is due to~$k'\ge-\tilde m$ and~$\mu\ge 0$.  So we
get~$-a\le c'\le\lambda^2c$ for some $\lambda\ge 1$, and we may conclude that
$c\ge-a$.
\end{proof}

\begin{proposition}
\label{prop:ample_cone}
Let~$X$ be a hyperkähler manifold of~$\KKK^{[m]}$-type or $\Kum_m$-type. Let~$g$ be one of the
canonical generators of~$D(H^2(X,\bZ))$. The Kähler cone $\cK_X$ is one of
the components of the positive cone cut out by the hyperplanes~$\kappa^\perp$,
for all classes $\kappa\in\NS(X)$ satisfying the following numerical condition:
writing~$\kappa^2=2l$,~$\kappa_*=kg$ with~$0\le k\le \tilde m$,
and~$d=\gcd(2\tilde m,k)$, then
\begin{equation}\label{eq:wall}
\begin{cases}
l=c\left(\frac{2m-2}d\right)^2-(m-1)\left(\frac kd\right)^2\text{ for an
integer }-1\le c<\frac{k^2}{4(m-1)}&\text{if $X$ is of $\KKK^{[m]}$-type};\\
l=c\left(\frac{2m+2}d\right)^2-(m+1)\left(\frac kd\right)^2\text{ for an
integer }0\le c<\frac{k^2}{4(m+1)}&\text{if $X$ is of $\Kum_m$-type}.
\end{cases}
\end{equation}
\end{proposition}
\begin{proof}
The $\KKK^{[m]}$-case is obtained by combining Theorem~\ref{thm:BHT} and
Proposition~\ref{prop:kappa} for $a=1$,
and the upper bound for $c$ comes from $\kappa^2=2l<0$. For the $\Kum_m$-case,
we use $a=0$ and we note that $k$ cannot take the value 0 because $l$ needs to
be negative.
So we will only consider $\kappa$ with $1\le k\le \tilde m = m+1$, and for such $\kappa$
we indeed get an element $s$ with $s^2\ge 0$ and $0<|(s,v)|=|{-k}|\le \tilde m
= m+1$.
\end{proof}

\begin{remark}\leavevmode
\begin{itemize}
\item To enumerate all the wall divisors, we let~$k$ run from~$0$ to~$\tilde m$
and for each~$k$, we let~$c$ run from~$-1$ or $0$
to~$\left\lceil\frac{k^2}{4\tilde m}\right\rceil-1$ to get the corresponding~$l$.
\item As an example, for~$\KKK^{[2]}$-type, the pair $(k,l)$ has three possibilities:
$(0,-1)$, $(1,-5)$, and $(1,-1)$. Thus we get $\kappa^2=-2$ and
$\div(\kappa)=1,2$, or
$\kappa^2=-10$ and $\div(\kappa)=2$. This was first conjectured in
\cite{HassettTschinkel}.
See also \cite{Mongardi:note}, where the cases of $\KKK^{[m]}$-type for $m\le4$ are worked
out; and \cite{HassettTschinkel:intersection}, where some examples for $\Kum_m$-type are given.
\item Analogous results for $\OG_6$ and $\OG_{10}$
are also established:
wall divisors on a hyperkähler manifold $X$ of~$\OG_6$-type are given by
elements $\kappa\in \NS(X)$ with $\kappa^2=-2$, or $\kappa^2=-4$ and
$\div(\kappa)=2$ \cite{MongardiRapagnetta};
wall divisors on a hyperkähler manifold $X$ of~$\OG_{10}$-type are given by
elements $\kappa\in \NS(X)$ with $0>\kappa^2\ge -4$, or $0>\kappa^2\ge -24$ and
$\div(\kappa)=3$ \cite{MongardiOnorati}.
\item In particular, the Kawamata--Morrison conjecture holds for all known
deformation types of hyperkähler manifolds by a result of Amerik--Verbitsky
\cite[Theorem~1.21]{AmerikVerbitsky}: for a given deformation type, since the
square of a wall class is bounded below, the automorphism group $\Aut(X)$ acts
on the set of faces of
$\cK_X$ with finitely many orbits.
\qedhere
\end{itemize}
\end{remark}

We will now describe the image of the period map.
Let~$\tau$ be the deformation type of a polarization, and
take an element~$h\in\tau$. For a vector~$u\in\Lambda$ with negative square and
linearly independent of~$h$, the hyperplane~$u^\perp\subset \bP(\Lambda_\bC)$
cuts a hyperplane section in the subset~$\Omega_h$ and induces a
divisor~$\cH_u$ in the period domain~$\cP_\tau=\Omega_h/\Mon(\Lambda,h)$ which
is called a \emph{Heegner divisor}. By abuse of notation, its image in~$\cP_T$
will also be denoted as~$\cH_u$.
\begin{proposition}
Take a deformation type of hyperkähler manifolds for which the
Kawamata--Morrison conjecture holds. Let $\tau$ be the deformation type of a
polarization and take $h\in \tau$. The complement of the image
of the period map~$\wp_\tau$ in~$\cP_\tau$ is the union of the Heegner
divisors~$\cH_\kappa$ induced by wall classes $\kappa\in\Lambda$ that are
orthogonal to~$h$.
\end{proposition}
\begin{proof}
For $x\in \kappa^\perp$, if there is a polarized pair~$(X,H)$ of deformation
type~$\tau$ such that~$\wp(X)=[x]\in\Omega_\marked$, take a marking~$\eta$ such
that~$\eta(H)=h$. Then~$\eta^{-1}(\kappa)$ is of type~$(1,1)$ and thus
algebraic. The class~$H$ is contained in the wall $\eta^{-1}(\kappa)^\perp$ and
thus not ample by the description of the Kähler cone, a contradiction.

Conversely, we consider a point~$[x]\in \Omega_h$ not belonging to any Heegner
divisor $\cH_\kappa$. By Proposition~\ref{prop:M_amp_dense}, we know that
$\cM^\amp_h$ can be identified as a dense open subset of $\Omega_h$ by the
marked period map $\wp$. If $[x]$ lies in this subset, then we know that $[x]$
is the period for some marked pair $(X,\eta)\in \cM^\amp_h$ for which
$\eta^{-1}(h)$ is ample; otherwise, since nefness is a closed condition, we can
choose $(X,\eta)$ so that $\eta^{-1}(h)$ is strictly nef, that is, it lies on
the boundary of the Kähler cone $\cK_X$. 
Since Kawamata--Morrison conjecture holds for $X$, we may conclude
that $\eta^{-1}(h)$ lies on a hyperplane $D^\perp$ for some wall divisor
$D\coloneqq\eta^{-1}(\kappa)$. But this means that the period $[x]$ is
contained in the Heegner divisor $\cH_\kappa$, where the wall class $\kappa$ is
orthogonal to $h$, and this is not the case by assumption.
\end{proof}

Finally we give a criterion for the existence of a wall class~$\kappa$
in~$h^\perp$ for $\KKK^{[m]}$-type and $\Kum_m$-type.
\begin{proposition}
\label{prop:kappainhperp}
For $\KKK^{[m]}$-type or $\Kum_m$-type,
let $h\in\Lambda$ be an element of divisibility $\gamma$. Let~$k$ and~$l$ be
integers satisfying the condition~\eqref{eq:wall}. Then there is a wall divisor
$\kappa\in h^\perp$ with $\kappa^2=2l$ and $\kappa_*=k\delta_*$ if and only if
we have $\gamma\mid k$. Equivalently, this is the condition~$\div h\cdot\div
\kappa\mid 2\tilde m$.
\end{proposition}

\begin{proof}
Recall that $\Lambda=\Lambda_0\oplus\bZ\delta$.
Write $h=\gamma ax+b\delta$ with $x\in\Lambda_0$ primitive, $\gcd(\gamma
a,2\tilde m)=\gamma$, and $\gcd(\gamma a,b)=1$. Write
\[
\kappa=\frac{2\tilde m}{d}(y+e\delta)+\frac{k}{d}\delta,
\]
with $y\in\Lambda_0$. Thus $\kappa$ being orthogonal to $h$ is
equivalent to
\[
\gamma a(x,y)=b(2\tilde me+k).
\]
Since $\gcd(\gamma,b)=1$ and
$\gamma\mid2\tilde m$, the condition $\gamma \mid k$ is clearly necessary.
Conversely, if this condition is met, we show that there exist a
suitable vector~$y$ and an integer~$e$ that give the desired~$\kappa$.
Since~$\gcd(\gamma a,2\tilde m)=\gamma$, we may choose~$e$ such that
\[
a\left|\frac{2\tilde m}{\gamma}e+\frac{k}{\gamma}\right..
\]
Thus we only need to find $y\in\Lambda_0$ with required $y^2$ and
$(x,y)$. By Eichler's criterion, this can be done by taking $\phi\in
\O(\Lambda_0)$ such
that $\phi(x)=u_1'+\frac{x^2}2u_2'$ and then choosing $y$ such that
$\phi(y)=(x,y)u_2'+u_1''+\frac{y^2}{2}u_2''$, where $\inner{u_1',u_2'}$
and~$\inner{u_1'',u_2''}$ are two copies of hyperbolic plane~$U$ in $\Lambda_0$.
\end{proof}
In the proof, since we have explicitly described the classes~$h$ and~$\kappa$,
if we look at the sublattice~$\inner{h,\kappa,v}$ in~$\tL$, its saturation
is generated by the three classes~$\frac{h-bv}\gamma$,~$s=\frac{d\kappa-k
v}{2\tilde m}$, and~$v$. So for this particular choice of~$\kappa$, the discriminant
of the saturation is~$\left|\frac{2d^2nl}{\gamma^2\tilde m}\right|$, while in general the
discriminant would be this number divided by some square.  Since the
Mukai lattice~$\tL$ is unimodular, this is also the discriminant of the
orthogonal~$\inner{h,\kappa,v}^\perp$, which can be identified with the
orthogonal~$\inner{h,\kappa}^\perp$ in~$\Lambda$. The latter is called the
\emph{transcendental lattice} of the Heegner divisor~$\cH_\kappa$.  Its
discriminant is also referred to as the discriminant of the Heegner
divisor~$\cH_\kappa$. Therefore we have the following corollary.
\begin{corollary}
Let~$T$ be a polarization type of square~$2n$ and divisibility~$\gamma$ on
hyperkähler manifolds of~$\KKK^{[m]}$-type or $\Kum_m$-type.  Let~$k$ and~$l$ be
integers satisfying the condition \eqref{eq:wall} (which only depends on~$m$)
such that~$\gamma\mid k$. For each connected component~$\cM_\tau$ of~$\cM_T$,
the period map~$\wp_\tau$ avoids at least one irreducible Heegner
divisor~$\cH_\kappa$ of discriminant~$\left|
\frac{2d^2nl}{\gamma^2\tilde m}\right|$ in~$\cP_\tau$, where~$d=\gcd(2\tilde m,k)$.
\end{corollary}
For example, for $\KKK^{[2]}$-type, we have already seen that~$(k,l)$ can
be~$(0,-1)$,~$(1,-5)$, and~$(1,-1)$. For a polarization type~$T$ of
square~$2n$, if the divisibility~$\gamma$ is equal to~$2$, the only possible
case is~$(0,-1)$ and we get a Heegner divisor of discriminant~$2n$; if the
divisibility~$\gamma$ is equal to~$1$, the three cases are all present and we
get Heegner divisors of discriminant~$8n$,~$10n$, and~$2n$.  This result is
however not exhaustive, since the sublattice we used above to compute the
discriminant might still not be primitive in general, and the discriminant will
be divided by some square. For example, when~$\gamma=1$, by \cite[Theorem
6.1]{DebarreMacri} it is also possible to have a Heegner divisor of
discriminant~$2n/5$ in the complement.  Note also that there might be several
irreducible Heegner divisors with the same discriminant while we have only
obtained one of them.

Another simple example works for almost every polarization type~$T$:
\begin{itemize}
\item If~$\gamma\le \tilde m$ we may take~$(k,l)$ to be~$(\gamma,-\tilde m)$
(and $c=0$), so the
discriminant is equal to~$2n$. In other words, for such a polarization
type~$T$, the restriction of the period map to every connected component
of~$\cM_T$ will avoid an irreducible Heegner divisor of discriminant~$2n$ in
the period domain
\item For a polarization type~$T$ not satisfying~$\gamma\le \tilde m$,~$\gamma$
is necessarily equal to the maximal value~$2\tilde m$. For $\KKK^{[m]}$-type, we may
take~$(k,l)$ to be~$(0,-1)$ (so $c=-1$), and the discriminant is then equal
to~$\frac{2n}{m-1}$, so we get a similar conclusion.
On the other hand, for $\Kum_m$-type, a polarization type $T$ of maximal
divisibility $2\tilde m=2(m+1)$ admits no orthogonal wall divisor:
since by Proposition~\ref{prop:kappainhperp} we must have $k=0$, so there exists
no $(k,l)$ satisfying the condition~\eqref{eq:wall}.
\end{itemize}

\section{Two examples}\label{sec:examples}
Using the numerical condition \eqref{eq:wall}, we can now compare the images
by the period map of various components.  Recall the picture of the polarized
period map from \eqref{eq:polarized-period}.  We prove the following result
for $\KKK^{[m]}$-type. Clearly the same idea can be adapted to $\Kum_m$-type.
\begin{proposition}
Let $a$ be a positive integer.
\begin{itemize}
\item[(i)]
For hyperkähler manifolds of $\KKK^{[144^a+1]}$-type,
there is a unique polarization type~$T$ of square~$288$ and divisibility~$12$,
for which the polarized moduli space~$\cM_T$ has exactly two components,
with different images in~$\cP_T$ under the period map.
\item[(ii)]
For hyperkähler manifolds of $\KKK^{[6^a+1]}$-type,
there is a unique polarization type~$T$ of square~$2$ and divisibility~$1$,
for which the polarized moduli space~$\cM_T$ is connected. The group $G$ is
isomorphic to $\bZ/2\bZ$, and the image of the period map in~$\cP_\tau$ is
not~$G$-invariant above~$\cP_T$.\qedhere
\end{itemize}
\end{proposition}
\begin{proof}
For (i), we may check by Proposition~\ref{prop:numberofT} that such
polarization type is unique and the polarized moduli space~$\cM_T$ has exactly
two components.
Note that by Proposition~\ref{prop:numberoftau},~$\gamma=12$ is the smallest
divisibility for the moduli space~$\cM_T$ to have more than
one component.

As $D(\Lambda)=\bZ/(2\cdot 144^a)\bZ$ and $\rho(2\cdot 144^a)=2$, by
Lemma~\ref{lemma:OG} we have $\O(D(\Lambda))=\set{\pm1,\pm g}$.
For $h\in T$, the class~$h_*$ is of
order~$12$ in~$D(\Lambda)$. So for any $\phi\in \O(\Lambda,h)$, we have
$\chi(\phi)=1$ since~$1$ is the unique element in~$\O(D(\Lambda))$ that
is~$\equiv1\pmod{12}$. This shows that $\O(\Lambda,h)\subset \tO(\Lambda)$ and
consequently, the group~$\O^+(\Lambda,h)/\Mon(\Lambda,h)$ is trivial. In this
case, both period domains~$\cP_\tau$ are canonically isomorphic to~$\cP_T$.

Since~$\cM_T$ has two components, we may choose $h,h'\in T$ belonging to
different~$\Mon(\Lambda)$-orbits or equivalently,~$\hO(\Lambda)$-orbits, as we
have seen in the proof of Proposition~\ref{prop:numberoftau} that they are the same.
There exists~$\psi\in\O(\Lambda)\setminus \hO(\Lambda)$ such that $\psi(h)=h'$.
We may assume that $\chi(\psi)=g$. Consider the period domain~$\cP_T$
realized as the quotient~$\Omega_h/\O^+(\Lambda,h)$
or~$\Omega_{h'}/\O^+(\Lambda,h')$. The automorphism~$\psi$ induces an
identification between the two, which maps each Heegner divisor~$\cH_\kappa$
to~$\cH_{\psi(\kappa)}$.

We consider a wall class $\kappa\in h^\perp$ with square~$2l$ and
$\kappa_*=k\delta_*$. The class $\kappa'=\psi(\kappa)$ has the same square $2l$
while $\kappa'_*=k'\delta_*$ with $k'\equiv gk\pmod{2\cdot 144^a}$. For
$\kappa'$ to also define a wall class, we need
\[
c'=c+\frac{{k'}^2-k^2}{4\cdot 144^a}\ge-1
\]
to hold.  So the idea is to choose some suitable~$k,l$ for which this condition
fails.  We let~$k=12g_0$ such that $k\equiv 12g\pmod{2\cdot 144^a}$ (so $g_0$
is the residue of $g$ modulo $24\cdot 144^{a-1}$). Since $g\ne\pm 1$ in
$\O(D(\Lambda))$, $g_0$ cannot be $\pm 1$ hence we have $g_0^2>1$. Then we can
let $c=-1$ and find the value for $l$ using \eqref{eq:wall}.
By Proposition~\ref{prop:kappainhperp}, there exists indeed such a wall
class~$\kappa\in h^\perp$.
On the other hand, the choice of $k$ means $k'=12$, so
$c'=-1+\frac{12^2-12^2g_0^2}{4\cdot 144^a}<-1$, and $\kappa'$ is not a
wall class.  This shows that the same Heegner divisor inside~$\cP_T$ is avoided
by the period map for one component but not for the other. Thus their images
in~$\cP_T$ by the period map are not the same.

For (ii), once again we may verify by Proposition~\ref{prop:numberofT} that
there is a unique such polarization type~$T$ with one connected component.
And by Lemma~\ref{lemma:OG}, since $D(\Lambda)=\bZ/(2\cdot 6^a)\bZ$ and
$\rho(2\cdot 6^a)=2$, we have $\O(D(\Lambda))=\set{\pm1,\pm g}$.

Since this $\O(\Lambda)$-orbit is unique, we may take $h=u_1'+u_2'$, where
$\inner{u_1',u_2'}$ is a copy of $U$.
The group $\O(\Lambda,h)$ contains $\O(\Lambda,U)\coloneqq\setmid{\phi\in
\O(\Lambda)}{\phi|_U=\Id}$ as a subgroup,
which is isomorphic to~$\O(U^\perp)$ since~$U$ is a direct summand.
Moreover, the inclusion~$\O(U^\perp)\simeq\O(\Lambda,U)\into\O(\Lambda)$
induces an isometry between the two discriminant groups. We use
Proposition~\ref{prop:chisurj} on $\O(U^\perp)$ to deduce that the
homomorphism~$\chi\colon\O(\Lambda)\to \O(D(\Lambda))$ when restricted
to~$\O(\Lambda,U)$, is still surjective. In particular, there is
$\phi\in\O(\Lambda,h)$ such that $\chi(\phi)=g$.
On the other hand, following the proof of Proposition~\ref{prop:eichler_SO},
there is an element $R\in \O(\Lambda, h)$ such that $\sigma(R)=-1$ and
$\chi(R)=1$. Let~$\psi$ be~$\phi$ if~$\sigma(\phi)$, and~$R\circ \phi$ otherwise.
Then~$\psi$ is in~$\O^+(\Lambda,h)$ with~$\chi(\psi)=g$.  Consequently, the
group~$\O^+(\Lambda,h)/\Mon(\Lambda,h)$ is isomorphic to~$\bZ/2\bZ$.

As in the previous case, we consider a wall class $\kappa\in h^\perp$ with
square $2l$ and $\kappa_*=k\delta_*$, for $k=g$ and $c=-1$. Such a class
exists by Proposition~\ref{prop:kappainhperp}.
However, the class $\kappa'=\psi(\kappa)$ will have $k'=1$, so
$c'=-1+\frac{1^2-g^2}{4\cdot 6^a}<-1$ and $\kappa'$ is not a wall class. This
shows that there are two Heegner divisors in $\cP_\tau$ that can be mapped to
each other under the action
of~$\O^+(\Lambda,h)/\Mon(\Lambda,h)$, but one is avoided by the period map and
the other is not. Thus we see in particular that the group~$G$ is non-trivial
and therefore also isomorphic to~$\bZ/2\bZ$, and the image of the period map is
not~$G$-invariant.
\end{proof}

\bibliographystyle{amsalpha}
\bibliography{refs}
\end{document}